\documentclass{amsart}
\usepackage{graphicx} 

\usepackage{amssymb}
\usepackage{amsthm}
\usepackage{amsmath}
\usepackage{mathrsfs}
\usepackage{mathtools}
\usepackage{amsbsy}
\usepackage[all]{xy}
\usepackage{bm}
\usepackage{hyperref}
\usepackage{tikz}
\usepackage{array}
\usepackage{colortbl,xcolor}
\usepackage{ytableau}
\usepackage{hhline}
\usepackage{graphicx, wrapfig}
\usepackage[margin=1in]{geometry}
\usepackage{tikz-cd}

\allowdisplaybreaks

\newtheorem{theorem}{Theorem}[section]
\newtheorem*{theorem*}{Theorem}
\newtheorem{proposition}[theorem]{Proposition}
\newtheorem*{proposition*}{Proposition}

\newtheorem*{corollary*}{Corollary}

\newtheorem{lemma}[theorem]{Lemma}

\newtheorem*{definition*}{Definition}

\newcommand{\NN}{\mathbb N}
\newcommand{\ZZ}{\mathbb Z}
\newcommand{\QQ}{\mathbb Q}

\newcommand{\CC}{\mathbb C}
\newcommand{\PP}{\mathbb P}

\newcommand{\gf}{\mathfrak{g}}

\newcommand{\Gr}{\mathcal{G}r}
\newcommand{\Oc}{\mathcal O}
\newcommand{\Kc}{\mathcal K}
\newcommand{\Pc}{\mathcal P}
\newcommand{\Fc}{\mathcal F}
\newcommand{\Zc}{\mathcal Z}
\newcommand{\Sc}{\mathcal S}
\newcommand{\Ls}{\mathscr L}
\newcommand{\IC}{\mathbf{IC}}
\newcommand{\gl}{\mathfrak{gl}}

\newcommand{\glinf}{\mathfrak{gl}_{\infty}^+}

\newcommand{\lam}{\lambda}
\newcommand{\eps}{\varepsilon}

\newcommand{\sgn}{\text{sgn}}

\newcommand{\Rep}{\operatorname{Rep}}
\newcommand{\Sym}{\operatorname{Sym}}
\newcommand{\End}{\operatorname{End}}
\newcommand{\Mod}{\operatorname{Mod}}
\newcommand{\xs}{\underline{x}}
\newcommand{\ts}{\underline{t}}

\title{Equivariant Schubert calculus and geometric Satake}
\author{Antoine Labelle}
\address{Department of Mathematics and Statistics \\ McGill University}
\email{antoine.labelle@mail.mcgill.ca}

\begin{document}

\maketitle

\begin{abstract}
    The main classical result of Schubert calculus is that multiplication rules for the basis of Schubert cycles inside the cohomology ring of the Grassmannian $G(n,m)$ are the same as multiplication rules for the basis of Schur polynomials in the ring of symmetric polynomials. In this paper, we explain how to recover this somewhat mysterious connection by using the geometric Satake correspondence to put the structure of a representation of $GL_m$ on $H^\bullet(G(n,m))$ and comparing it to the Fock space representation on symmetric polynomials.  This proof also extends to equivariant Schubert calculus, and gives an explanation of the relationship between torus-equivariant cohomology of Grassmannians and double Schur polynomials.
\end{abstract}

\section{Introduction}

The cohomology ring of the Grassmannian $G(n,m)$ of $n$-dimensional subspaces of $\CC^m$ has a natural basis given by Schubert cycles $\sigma_\lam$, which comes from a cell decomposition of $G(n,m)$ into Schubert cells. Schubert calculus on Grassmannians is the problem of understanding the multiplication in the cohomology ring with respect to this basis, i.e. to calculate the structure constants $c_{\lam \mu}^\nu$ in the expansion
\begin{equation*} \label{eq:structure-consts}
    \sigma_\lam \cdot \sigma_\mu = \sum_\nu c_{\lam \mu}^\nu \sigma_\nu.
\end{equation*}
It is well-known that these constants can be identified as the Littlewood-Richardson coefficients, which also describe the multiplication of symmetric polynomials in the basis of Schur polynomials. In other words, there is a ring isomorphism from a certain quotient of the ring of symmetric polynomials in $n$ variables to $H^\bullet(G(n,m))$, which sends Schur polynomials to Schubert cycles (an observation first due to Lascoux). The classical proof of this result goes by first considering the case of multiplication by special Schubert cycles, described by Pieri's rule, and then expressing a general Schubert cycle in terms of special ones via Giambelli's formula, which mirrors the Jacobi-Trudi formula for Schur polynomials \cite[\S 14.7]{Ful}. This unfortunately does not give a very conceptual explanation for the existence of the isomorphism. 

In this note, we give an alternative explanation, by showing how this ring isomorphism can be obtained from an isomorphism of $GL_m$-representations between the two sides. Indeed, $H^\bullet(G(n,m))$, can be given the structure of a $GL_m$ representation via the celebrated geometric Satake correspondence. It is known from the general theory to be isomorphic to the highest weight representation of weight $\omega_n= (\underbrace{1, \ldots, 1}_{n}, 0, \ldots, 0)$, which is the wedge power $\bigwedge^n \ZZ^m$ \footnote{Note that we could guess the isomorphism $H^\bullet(G(n,m)) \cong \bigwedge^n \ZZ^m$ as $\ZZ$-modules without knowing the geometric Satake correspondence, as both sides are free of rank $\binom{m}{n}$. One can therefore just match bases and transfer the representation structure on $\bigwedge^n \ZZ^m$ to make $H^\bullet(G(n,m))$ into a representation of $GL_n$, as in \cite{GSal}. The power of the Satake correspondence is to give a conceptual geometric definition of this representation structure on $H^\bullet(G(n,m))$.}.
On the other hand, there is also the well-known Fock space representation of $\gl_\infty$ on the ring of symmetric functions \cite{Tin}, which becomes a representation of $GL_m$ after taking a suitable quotient and is also isomorphic to the wedge representation. It is then possible to show that this representation isomorphism is in fact a ring isomorphism by looking at the action of a certain subalgebra of $\gl_m$ which is related to the cohomology ring of the affine Grassmannian. It also follows immediately that, up to a scalar, the isomorphism takes Schubert cycles to Schur polynomials, since both are weight bases for their respective representations. If we work over $\ZZ$, this scalar is just a sign $\pm 1$. Showing that the sign is in fact always positive is a somewhat delicate issue. To do this, we need as a geometric input the fact that the structure constants for multiplication of Schubert classes are nonnegative.

Moreover, this whole story generalizes perfectly well to $T$-equivariant cohomology, where $T\subset GL_m$ is the torus of diagonal matrices, given that we work over the ring $H_T^\bullet(\text{pt})=\ZZ[t_1, t_2, \ldots, t_m]$. Schur polynomials, in that case, get replaced by so-called \emph{double Schur polynomials}, which are symmetric polynomials with coefficients in $\ZZ[t_1, t_2, \ldots]$. Therefore, we work equivariantly in the rest of the paper for the sake of generality, but the non-equivariant case can be recovered immediately by setting $t_i=0$ for $i=1, \ldots, m$.  

The goal of this paper is partly expository, as the idea of studying Schubert calculus with the help of the geometric Satake action of $GL_m$ already appears in the litterature. We refer the reader in particular to \cite{AN}, which investigates in details the relationship between Schubert calculus and the geometric Satake correspondence, including applications to Schubert calculus on orthogonal Grassmannians and quantum Schubert calculus. The idea of putting a representation structure on the cohomology ring of the Grassmannian also appears in earlier papers such as \cite{Gat, GSal, GSan, Lak2, Lak, LT2, LT} and applications of these ideas to the representation theory of the Lie algebra of infinite size matrices can be found in \cite{BCGM}.

In Section \ref{sec:geometric-satake}, we recall the geometric Satake correspondence of Mirkovic and Vilonen and explain how it defines a representation structure on $H^\bullet(G(n,m))$. Then, in Section \ref{sec:equiv-coh}, we upgrade this to a representation structure on $H_T^\bullet(G(n,m))$ by applying the correspondence with coefficients in $H_T^\bullet(\text{pt})$. In Section \ref{sec:double-schur}, we put a representation structure on the ring of symmetric polynomials and define double Schur polynomials, which form the weight basis of this representation. We also define a certain quotient of the ring of symmetric polynomials, depending on $m$. In Section \ref{sec:rep-iso}, we note that this quotient is isomorphic, as a representation of $GL_m$ over $H_T^\bullet(\text{pt})$, to $H_T^\bullet(G(n,m))$. We then explain in Section \ref{sec:ring-struct} how to upgrade this representation isomorphism into a ring isomorphism, by studying the action of the equivariant cohomology ring of the affine Grassmannian. Finally, in Section \ref{sec:sgn}, we tackle the problem of showing that the sign by which Schubert cycles and double Schur polynomials differ is in fact always $1$, by using an important positivity result due to Graham \cite{Graham} for multiplication of equivariant Schubert classes.

\section{The geometric Satake correspondence} \label{sec:geometric-satake}

In this section, we briefly summarize the geometric Satake correspondence of Mirkovic and Vilonen \cite{MV}. For a detailed exposition of the topic, see \cite{BR}.

\subsection{The affine Grassmannian}

Let $G$ be a complex reductive group and $T\subset B$ a maximal torus and Borel subgroup in $G$. Let $N$ be the unipotent radical of $B$. Set $\Kc = \CC(\!(t)\!)$ and $\Oc = \CC[\![t]\!]$.

The affine Grassmannian $\Gr$ is an ind-variety whose set of $\CC$-points is the quotient $G(\Kc)/G(\Oc)$. Any coweight $\mu \in X_*(T)$ determines an element of $G(\Kc)$, hence a point of $\Gr$, denoted $L_\mu$. There is a stratification by $G(\Oc)$ orbits
\[\Gr = \bigsqcup_{\mu \in X^+_*(T)} \Gr^\mu \qquad \text{where} \qquad \Gr^\mu = G(\Oc) \cdot L_\mu\]
 and $X^+_*(T)$ is the set of dominant coweights. Moreover, we have $\overline{\Gr^\mu} = \bigsqcup_{\nu \le \mu} \Gr^\nu$. 
 Let $\Pc$ be the category of $G(\Oc)$-equivariant perverse sheaves on $\Gr$ with coefficients in $\ZZ$ (by \cite[Proposition 2.1]{MV}, this is equivalent under the forgetful functor to the category of perverse sheaves constructible with respect to the stratification above). For each $\mu \in X^+_*(T)$, there is a perverse sheaf $\IC_\mu$, supported on $\overline{\Gr^\mu}$, called the \emph{intersection cohomology sheaf}\footnote{If we take coefficients in a field $k$ rather than $\ZZ$, these intersection cohomology sheaves are exactly the simple objects in the category $\Pc(\Gr, k)$}. If $\mu$ is minimal with respect to the partial order on $X_*(T)$, so that $\overline{\Gr^\mu}=\Gr^\mu$, then this is simply the shifted constant sheaf $\underline{\ZZ}_{\Gr^\mu}[\dim \Gr^\mu]$. 

 Other important subspaces of $\Gr$ are the semi-infinite orbits $S_\mu=N(\Kc)\cdot L_\mu$, which are indexed by (not necessarily dominant) coweights $\mu$. We have 
$\Gr = \bigsqcup_{\mu \in X_*(T)} S_\mu$ and $\overline{S_\mu} = \bigsqcup_{\nu \le \mu} S_\nu$.

\subsection{The correspondence}
 
 There is a notion of convolution for perverse sheaves on $\Gr$, which makes $\Pc$ into a monoidal category. Hypercohomology defines a tensor functor $H^\bullet : \Pc \to \Mod_\ZZ$ to the category of finitely generated $\ZZ$-modules. By the Tannakian formalism, there is an equivalence of monoidal categories
    \[ \Pc \cong \Rep_\ZZ(G^\vee), \]
 which commutes with the natural functors to $\Mod_\ZZ$ on both sides, where $G^\vee$ is an algebraic group over $\ZZ$ such that $G^\vee(R)$ is the group of tensor automorphisms of the functor $H^\bullet(-)\otimes R : \Pc \to \Mod_R$ for any ring $R$. Moreover, $G^\vee$ can actually be identified with the Langlands dual group of $G$ over $\ZZ$, i.e. the unique split reductive group over $\ZZ$ whose root datum is dual to the root datum of $G$, and the equivalence sends $\IC_\mu$ to the Schur module of highest weight $\mu$ \cite[Proposition 13.1]{MV}.

Moreover, the fiber functor $H^\bullet(-)$ factors through $X_*(T)$-graded $\ZZ$-modules via the natural isomorphism
\[H^\bullet(-) = \bigoplus_{\mu \in X_*(T)} H_c^\bullet(S_\mu, -)\]
and $H_c^\bullet(S_\mu, -)$ is concentrated in degree $2\langle \rho, \mu\rangle$, where $\rho$ is the Weyl vector (half the sum of the positive roots). Then there is a canonical maximal torus $T^\vee \subset G^\vee$ which acts by $\mu$ on the degree $\mu$ part of $H^\bullet(\Fc)$ for any perverse sheaf $\Fc$.

\subsection{The case of $GL_m$}

We will be mainly interested in the special case $G = GL_m(\CC)$, where we take $T$ to be the standard torus of diagonal matrices and $B$ the subgroup of upper triangular matrices. In this case, there is a concrete interpretation of points of $\Gr$ as $\Oc$-lattices inside $\Kc^m$. Indeed, there is a transitive action of $GL_m(\Kc)$ on such lattices, and the stabilizer of the standard lattice $L_0=\Oc^m$ is exactly $GL_m(\Oc)$. 

By the theory of Smith normal form over PIDs \cite[Chapter 3]{Jacobson}, for any lattice $L$ there exists a basis $b_1, \ldots, b_m$ of $L_0$ and integers $\mu_1 \ge \ldots \ge \mu_m$ such that $t^{\mu_1} b_1, \ldots, t^{\mu_m} b_m$ form a basis for $L$. The tuple $(\mu_1, \ldots, \mu_m)$ is uniquely determined and is called the \emph{type} of $L$. Then, for every $\mu \in X^+_*(T)$, $\Gr^\mu$ is exactly the set of lattices of type $\mu$. 

If $\mu$ is the fundamental weight $\omega_n= (\underbrace{1, \ldots, 1}_{n}, 0, \ldots, 0)$, then lattices of type $\mu$ can be identified with $n$-dimensional quotients of $L_0/tL_0 \cong \CC^m$, so $\Gr^{\omega_n}=\overline{\Gr^{\omega_n}}=G(n,m)$ is the usual Grassmannian. In that case, $\IC_{\omega_n}$ is the shifted constant sheaf $\underline{\ZZ}_{G(n,m)}[\dim G(n,m)]$, so $H^\bullet(\Gr, \IC_{\omega_n}) = H^\bullet(G(n,m))$ (the classical singular cohomology) up to shift. 

The semi-infinite orbits that intersect $\Gr^{\omega_n}$ are exactly $S_\mu$ for $\mu$ a sequence of zeros and ones containing exactly $n$ ones. In that case, $S_\mu \cap \Gr^{\omega_n}$ is the $N(\CC)$-orbit of $L_\mu$, which is the opposite Schubert cell $\Omega_\mu$ indexed by $\mu$\footnote{This is an opposite Schubert cell rather than a Schubert cell because we identify $\Gr^{\omega_n}$ with the Grassmannian of $n$-dimensional \emph{quotients} of $L_0/tL_0$, or equivalently $n$-dimensional subspaces of $(L_0/tL_0)^*$. The passage to dual space makes $GL_m(\CC)$ act on $G(n,m)$ under this identification not in the usual way, but via the inverse transpose matrix.}.

The root system of $GL_m$ is self-dual, so $G^\vee$ is simply $GL_m$ over $\ZZ$. The geometric Satake correspondence then identifies $H^\bullet(G(n,m))$ with the Schur module of highest weight $\omega_n$ for $GL_m/\ZZ$, which is the wedge representation $\bigwedge^n \ZZ^m$. The identification of $G^\vee$ with $GL_m$ is not uniquely determined, but we can fix one by requiring that the isomorphism between $\ZZ^m$ and $H^\bullet(G(1,m))=H^\bullet(\PP^{m-1})$ sends the $k^\text{th}$ standard basis vector to the class of $\overline{\Omega_{e_k}}=\PP^{m-k}$, where $e_k=(0, \ldots, 1, \ldots, 0)$ with the one in position $k$.

\section{Equivariant cohomology} \label{sec:equiv-coh}

Let $R_T=H_T^\bullet(pt)=\Sym X^*(T)$ denote the $T$-equivariant cohomology of a point. In \cite{YZ}, it is proven that there is a natural isomorphism of functors $\Pc \to \Mod_{R_T}$ between $H^\bullet_T(-)$ and $H^\bullet(-)\otimes R_T$. This isomorphism comes from the decompositions

\[H^\bullet(-) = \bigoplus_{\mu \in X_*(T)} H_c^\bullet(S_\mu, -)\]
and
\[H_T^\bullet(-) = \bigoplus_{\mu \in X_*(T)} H_{T,c}^\bullet(S_\mu, -),\]
and canonical isomorphisms $H_{T,c}^\bullet(S_\mu, -) =H_c^\bullet(S_\mu, -)\otimes R_T$ due to the fact that $H_c^\bullet(S_\mu, -)$ is concentrated in one degree. Therefore, the group of tensor automorphisms of the functor $H_T^\bullet(-)$ can be identitifed by the Tannakian formalism with $G^\vee(R_T)$ and the equivariant cohomology of any perverse sheaf acquires the structure of a representation of $G^\vee$ over $R_T$.

Now, and for the rest of this note, take $G=GL_m(\CC)$ and $T$ the standard maximal torus. In this case, $R_T$ is the polynomial ring $\ZZ[t_1, \ldots, t_m]$. By the discussion above, the $R_T$-module $V:=H_T^\bullet(G(1,m))=H_T^\bullet(\PP^{m-1})$ gets identified with the standard representation of $GL_m(R_T)$, where the $k^{th}$ basis vector correspond to the class of $\overline{\Omega_{e_k}}=\PP^{m-k}$. It is a standard fact that $H_T^\bullet(\PP^{m-1})$ is isomorphic, as an $R_T$-algebra, to $R_T[x]/\prod_{i=1}^m (x+t_i)$ where $x=c_1^T(\Oc(1))$ and that, under this isomorphism, the class of $\overline{\Omega_{e_k}}$ correspond to $\prod_{i=1}^{k-1} (x+t_i)$ \cite[Chapter 4, Example 7.4]{AF}. To avoid a choice of basis, we can think of $G^\vee_{R_T}$ as the automorphism group of the free $R_T$ module $V$. Then we have an isomorphism of $G^\vee_{R_T}$ representations

\begin{equation} \label{eq:cohomology-eq-wedge}
    H_T^\bullet(G(n,m)) = \bigwedge\nolimits^n V.
\end{equation}

\section{Symmetric functions and double Schur polynomials} \label{sec:double-schur}

Let $R=\ZZ[t_1, t_2, \ldots]$ be a polynomial ring in infinitely many variables. Denote by $\Lambda_n, \Lambda_n^\sgn\subset R[x_1, \ldots, x_n]$ the $R$-modules of symmetric and skew-symmetric polynomials in $n$ variables with coefficients in $R$, respectively. 
Note that $\Lambda_n^{sgn}$ can be identified with $\bigwedge^n R[x]$ via
\[f_1(x) \wedge \cdots \wedge f_n(x) \mapsto \sum_{\sigma \in S_n} \sgn(\sigma) f_{\sigma(1)}(x_1)\cdots f_{\sigma(n)}(x_n).\]

The "double monomials" $(x|\ts)^k:=(x+t_1)\cdots(x+t_k)$ form an $R$-basis of $R[x]$ as $k$ runs over nonnegative integers. This gives an induced basis $(a_\nu(\xs|\ts))$ of $\Lambda_n^\sgn$ indexed by strictly decreasing sequences $\nu_1> \ldots> \nu_n$ of nonnegative integers, where
\[ a_\nu(\xs|\ts)  = \sum_{\sigma \in S_n} \sgn(\sigma) \prod_{i=1}^n (x_{\sigma(i)}|\ts)^{\nu_i}\]
Let $\rho = (n-1, n-2, \ldots, 1, 0)$. We can interpret the definition of $a_\rho$ as the determinant $\det((x_i|\ts)^{n-j})_{i,j=1}^n$ and then do row operations to reduce it to the usual Vandermonde determinant $\det(x_i^{n-j})_{i,j=1}^n = \prod_{1\le i<j \le n} (x_i-x_j)$. Therefore we have

\[a_\rho(\xs| \ts) = \prod_{1\le i<j \le n} (x_i-x_j).\]

\begin{proposition} \label{prop:vandermonde-mul-iso}
    Multiplication by $a_\rho(\xs|\ts)$ defines an isomorphism of $R$-modules $\Lambda_n \to \Lambda_n^{sgn}$.
\end{proposition}

\begin{proof}
    Since the product of a skew-symmetric polynomial with a symmetric polynomial is clearly skew-symmetric, multiplication by $a_\rho$ does send $\Lambda_n$ to $\Lambda_n^{sgn}$, and the map is injective since $R[x_1,\ldots, x_n]$ is an integral domain. Moreover, every skew-symmetric polynomial vanishes whenever two variables are equal, hence is divisible by $x_i-x_j$ for all $i<j$, hence is divisible by $a_\rho$. This shows the surjectivity.
\end{proof}

By the proposition above, we get an induced basis $(s_\lam(\xs| \ts))$ of $\Lambda_n$ indexed by partitions $\lam_1 \ge \ldots \ge \lam_n\ge 0$, where

\[s_\lam(\xs| \ts) = \frac{a_{\lam+\rho}(\xs|\ts)}{a_\rho(\xs| \ts)}.\]
These polynomials are called \emph{double Schur polynomials} (or, in some sources, \emph{factorial Schur polynomials}) \cite{Mihalcea} \footnote{Our definition differ slightly from the usual convention for double Schur polynomials, which replaces $t_i$ by $-t_i$ in the definition of double monomials.}.

Let $\glinf(R)$ be the Lie algebra of $\NN \times \NN$ matrices over $R$ with finitely many nonzero entries. It acts on the free $R$-module with basis indexed by $\NN$, which we can identify with $R[x]$ via the basis $((x|\ts)^k)_{k\in \NN}$. Hence it also acts on $\Lambda_n$ through the identification
\[\Lambda_n \underset{\sim}{\overset{a_\rho}{\longrightarrow}}\Lambda_n^\sgn \cong \bigwedge\nolimits^n R[x].\]
This is an equivariant version of the  so-called \emph{Fock space} representation of $\glinf(\CC)$ on the algebra of symmetric functions\footnote{The term Fock space actually usually refers to the representation of the Lie algebra $\gl_\infty$ of $\ZZ \times\ZZ$ matrices on symmetric functions with infinitely many variables, which is obtained as a limit of our representation as $n \to \infty$}. It is clear from the construction that the double Schur polynomials are then a weight basis of this representation for the action of the subalgebra of diagonal matrices (and this uniquely characterize double Schur polynomials up to sign, because all weights have multiplicity one).

\subsection{Truncated versions} \label{sec:truncated}

Let $J_m$ be the ideal of $R[x_1, \ldots, x_n]$ generated by $t_{m+1}, t_{m+2}, \ldots$ and $(x_i|\ts)^m$ for $i=1, \ldots, n$, and $I_m^\sgn= J_m \cap \Lambda_n^\sgn$. As an $R$-submodule of $\Lambda_n^\sgn$, $I_m^\sgn$ is clearly generated by $t_{m+1}, t_{m+2}, \ldots$ and the basis elements $a_\nu$ for strict partitions $\nu$ with $\nu_1\ge m$. 

Note that $I_m^\sgn$ is in fact a sub-$\Lambda_n$-module of $\Lambda_n^\sgn$ (since $J_m$ is an ideal). Hence, if we let $I_m$ be the preimage of $I_m^\sgn$ under the "multiplication by $a_\rho$" isomorphism $\Lambda_n \overset{\sim}{\to} \Lambda_n^\sgn$, then $I_m$ is an ideal of $\Lambda_n$. Note that $I_m$ is generated as an $R$-module by $t_{m+1}, t_{m+2}, \ldots$ and the basis elements $s_\lam$ for partitions $\lam$ with $\lam_1> m-n$. 

Define $\Lambda_{n,m}^\sgn = \Lambda_n^\sgn/I_m^\sgn$ and $\Lambda_{n,m} = \Lambda_n/I_m$. Identifying $R_T$ with $R/(t_{m+1}, t_{m+2}, \ldots)= \ZZ[t_1, \ldots, t_m]$, we see that $\Lambda_{n,m}^\sgn$ $\Lambda_{n,m}$ have the structure of $R_T$-modules, and multiplication by $a_\rho$ defines an isomorphism between them.

The map $R[x]\to V=R_T[x]/(x|\ts)^m$ which quotients by $t_{m+1}, t_{m+2}, \ldots$ and by $(x|\ts)^m$ induces a surjection

\begin{equation}
\Lambda_n^\sgn = \bigwedge\nolimits^n R[x] \twoheadrightarrow \bigwedge\nolimits^n V
\end{equation}

which has kernel $I_m^\sgn$, so we get an identification of $\Lambda_{n,m}^\sgn$ with $\bigwedge\nolimits^n V$ and $\Lambda_{n,m}^\sgn$ gets the structure of a representation of $G^\vee_{R_T}=\underline{\operatorname{Aut}}_{R_T}(V)$. We can also transfer this representation structure to $\Lambda_{n,m}$ via the "multiplication by $a_\rho$" isomorphism. Then the quotient map 

\[\Lambda_n \twoheadrightarrow \Lambda_{m,n}\]

is compatible with the representation structure on both sides (in the sense that it is $\gl_m(R)$ equivariant after making both sides into representations of $\gl_m(R)$ via the quotient map $\gl_m(R) \to \gl_m(R_T) \cong \operatorname{Lie}(G^\vee_{R_T})$ and the inclusion $\gl_m(R)\hookrightarrow \glinf(R)$).

\section{The representation isomorphism} \label{sec:rep-iso}

We have put the structure of a representation of $G^\vee_{R_T}$ on $H_T^\bullet(G(n,m))$ and $\Lambda_{n,m}$, and saw that both are isomorphic to the $n^{th}$ wedge representation. We therefore have an isomorphism

\[\Phi : \Lambda_{n,m} \to H_T^\bullet(G(n,m))\]
of $G^\vee_{R_T}$-representations. 

There is a bijection between partitions that fit in an $n$ by $m-n$ rectangle and coweights $\mu \in X_*(T)$ having $n$ ones and zeros everywhere else, which sends a partition $\lam$ to the coweight $\mu$ with ones in positions $\lam_1 + n, \lam_2 +n-1, \ldots, \lam_n+1$. Since $(x|\ts)^k$ span the weight $e_{k+1}$ subspace of $V$, we see that the weight $\mu$ subspace of $\Lambda_{n,m}$ is spanned by $s_\lam$, where $\lam$ corresponds to $\mu$ under the bijection above. On the other hand, we also know by the discussion of Section \ref{sec:equiv-coh} that the weight $\mu$ subspace of $H_T^\bullet(G(n,m)$ is $H_{T,c}^\bullet(S_\mu, \underline{\ZZ}_{G(n,m)}[\dim G(n,m)])$, which is spanned by the equivariant cohomology class of $\overline{\Omega_\mu}$. We will denote this class by $\sigma_\lam$, where again $\lam$ corresponds to $\mu$ under the above bijection.

Since representation isomorphisms preserve weight spaces, and $\pm 1$ are the only units in $R_T$, it follows that $\Phi$ sends $s_\lam$ to $\pm \sigma_\lam$ for all $\lam$.

Note also that, by Schur's lemma, $\Phi$ is uniquely determined up to a unit in $R_T$, i.e. up to a sign. We can fix this sign by requiring that $\Phi$ preserves multiplicative identities, i.e. $\Phi(s_\varnothing)=\sigma_\varnothing$.

\section{Recovering the ring structure} \label{sec:ring-struct}

\subsection{Equivariant cohomology of $\Gr$}

By the Tannakian formalism \cite{DM}, the Lie algebra $\gf^\vee_{R_T}$ of $G^\vee$ over $R_T$ can be identified with the natural endomorphisms $(\phi_\Fc)_{\Fc\in\Pc}$ of the fiber functor $H_T^\bullet(-) : \Pc \to \Mod_{R_T}$ that satisfy

\begin{equation} \label{eq:lie-alg-condition}
\phi_{\Fc_1 * \Fc_2}=1 \otimes \phi_{\Fc_1} + \phi_{\Fc_1} \otimes 1 
\end{equation}
\cite[5.3]{YZ}. 

We now consider the cohomology ring $H_T^\bullet(\Gr)$. Since $\Gr$ is homeomorphic to the group of based polynomial loops in $G$, there is a natural Hopf algebra structure on this ring. Any element of $H_T^\bullet(\Gr)$ defines by cup-product a natural endomorphism of $H_T^\bullet(-)$, and, because of \cite[Propostion 2.7]{YZ}, equation (\ref{eq:lie-alg-condition}) is satisfied for primitive elements of $H_T^\bullet(\Gr)$, so we have a map

\begin{equation} \label{eq:cohomology-to-lie-alg}
    H^\bullet_T(\Gr)^\text{prim} \to \gf_{R_T}^\vee.
\end{equation}

Consider the equivariant line bundle $\Ls$ on $\Gr$ whose fiber over a lattice $L$ is 

\[\det\left(\frac{L}{L\cap L_0} \right) \otimes \det\left(\frac{L_0}{L\cap L_0} \right)^*,\]
where $\det$ of a finite dimensional $\CC$-vector space means the top exterior power.
By \cite[Lemma 5.1]{YZ}, the element $e_T=c_1^T(\Ls)\in H^2_T(\Gr)$ is primitive, hence defines an element of the Lie algebra $\gf_{R_T}^\vee$, which we also denote by $e_T$.

If we identify $G^\vee_{R_T}$ with the automorphism group of the $R_T$ -module $V$, then its Lie algebra $\gf^\vee_{R_T}$ can be identified with $\End_{R_T}(V)$.

\begin{proposition}
    Under the identification $\gf_{R_T}^\vee=\End_{R_T}(V)$, $e_T$ is the multiplication by $-x$ operator.
\end{proposition}

\begin{proof}
    The line bundle $\Ls$ restricts to the tautological line bundle $\Oc(-1)$ on $\PP^{m-1}$, so $e_T=c_1^T(\Ls)$ acts via multiplication by $c_1^T(\Oc(-1))=-c_1^T(\Oc(1))=-x$ on $V=H_T^\bullet(\PP^{m-1})$.
\end{proof}

Since $H^\bullet_T(\Gr)$ is commutative, the image of (\ref{eq:cohomology-to-lie-alg}) must lie inside the centralizer $\Zc_{e_T}$ of $e_T$ inside $\End_{R_T}(V)$. Since $-x$ generates $V$ as an $R_T$-algebra, the $R_T$-linear endomorphisms of $V$ that commute with multiplication by $-x$ are exactly the operators of multiplication by some element of $V$. Hence we can identify $\Zc_{e_T}$ with $V$, acting on itself by multiplication.

In fact, after changing coefficients to a field, it turns out that (\ref{eq:cohomology-to-lie-alg}) is an isomorphism onto $\Zc_{e_T}$ \cite[Corollary 5.3.2]{Gin}\footnote{The paper \cite{Gin} works with $\CC$ coefficients, but it is clear that being an isomorphism over $\QQ$ or over $\CC$ is equivalent.}:

\begin{theorem} \label{thm:prim-iso-centralizer}
    The natural map 
        \[ H^\bullet_T(\Gr, \QQ)^\text{prim} \to \Zc_{e_T} \otimes \QQ\]
    is an isomorphism, and $H^\bullet_T(\Gr,\QQ)^\text{prim}$ freely generates $H^\bullet_T(\Gr,\QQ)$, so there is an induced isomorphism 
    \[ H^\bullet_T(\Gr, \QQ) \to U(\Zc_{e_T} \otimes \QQ).\]
\end{theorem}

\subsection{The main theorem}

We now have the tools to prove that $\Phi$ preserves multiplication.

\begin{theorem} \label{thm:ring-hom}
    The isomorphism of representations $\Phi : \Lambda_{n,m} \to H_T^\bullet(G(n,m))$ is a ring isomorphism.
\end{theorem}

\begin{proof}
    First, since both sides are free $\ZZ$-modules, it's enough to show that $\Phi$ is a ring homomorphism after tensoring with $\QQ$. Recall also that we chose $\Phi$ so that it preserves the multiplicative identities.
    
    Since $\Phi$ is $\gf_{R_T}^\vee$-equivariant, it is in particular $\Zc_{e_T}$-equivariant. Moreover, we claim that any $f\in \Zc_{e_T}$ acts on both the domain and codomain of $\Phi$ via multiplication by $f\cdot 1$ (using the ring structure on the domain and codomain of $\Phi$). For $H_T^\bullet(G(n,m))$, this is because of Theorem \ref{thm:prim-iso-centralizer} and the fact that a class $\alpha \in  H^\bullet_T(\Gr)$ acts via multiplication by $\alpha|_{G(n,m)}$. For $\Lambda_{n,m}$, we see from the rule \[X \cdot (v_1 \wedge \ldots \wedge v_n) = \sum_{i=1}^n v_1 \wedge \ldots \wedge Xv_i \wedge \ldots \wedge v_n \]
    for the action of a Lie algebra on wedge powers that, for any polynomial $f\in V\cong \Zc_{e_T}$, $f$ acts on $\Lambda_{n,m}^\sgn$ (hence also on $\Lambda_{n,m}$) via multiplication by $f(x_1)+\ldots+f(x_n)$.
    
    We have therefore (after tensoring with $\QQ$) a commutative diagram
    \begin{center}
    \begin{tikzcd}
                                             & {U(\Zc_{e_T}\otimes \QQ)} \arrow[ldd] \arrow[rdd] &                       \\
                                             &                                   &                       \\
    {\Lambda_{n,m}}\otimes \QQ \arrow[rr, "\Phi_\QQ"] &                                   & {H_T^\bullet(G(n,m),\QQ)}
    \end{tikzcd}
    \end{center}
    where $U(\Zc_{e_T}\otimes \QQ)=\Sym_{R_T\otimes \QQ}^\bullet (\Zc_{e_T}\otimes \QQ)$ is the universal enveloping algebra of $\Zc_{e_T}\otimes \QQ$, seen as an abelian Lie algebra over $R_T\otimes \QQ$, and the two oblique arrows are given by $f \mapsto f \cdot 1$.

    The two vertical arrows are ring homomorphisms, and the left vertical arrow is surjective because the power sum symmetric functions $x_1^k + \cdots + x_n^k$ generate $R[x_1,\ldots,x_n]^{S_N} \otimes \QQ$ (and therefore its quotient $\Lambda_{n,m}\otimes \QQ$) as an $R \otimes \QQ$-algebra. This implies that $\Phi_\QQ$ itself is a ring isomorphism.

\end{proof}

\section{Getting the signs right} \label{sec:sgn}

We constructed a ring isomorphism $\Phi : \Lambda_{n,m} \to H_T^\bullet(G(n,m))$ which sends $s_\lam(\xs| \ts)$ to $\pm \sigma_\lam$. We now tackle the delicate problem of proving that the sign is in fact always positive.

The main geometric input we need is the following positivity result for the multiplication of Schubert classes \cite[Chapter 19, Theorem 3.1]{AF}.

\begin{theorem} \label{thm:positivity}
    The structure constants $c_{\lam\mu}^\nu$ in the expansion
    \[\sigma_\lam \cdot \sigma_\nu = \sum_\nu c_{\lam\mu}^\nu \sigma_\nu\]
    lie in $\NN[t_1-t_2, \ldots, t_{n-1}-t_n]\subset R_T$.
\end{theorem}

We will prove that the signs are positive by considering the iterated action of $-e_T\in H^2(\Gr)^\text{prim} \subset \gf^\vee_{R_T}$. The following two lemmas describe more explicitly how $e_T$ acts on both sides of $\Phi$.

\begin{lemma} \label{lem:res-eT}
    The restriction of $e_T$ to $G(n,m)$ is $-\sigma_1+t_1+\ldots + t_n$.
\end{lemma}

\begin{proof}
    Note that the restriction of $\Ls$ to $G(n,m)$ is the top exterior power of the tautological vector bundle $\Sc$, so
    \[ e_T|_{G(n,m)} = c_1^T\left(\bigwedge\nolimits^n \Sc\right).\]
    On the other hand there is an equivariant map of line bundles
    \[ \bigwedge\nolimits^n \Sc \to \bigwedge\nolimits^n \Oc^n\]
    (where $\Oc$ denotes the trivial line bundle on $G(n,m)$) induced by the quotient map $\Oc^m \to \Oc^n$ which forgets the last $m-n$ coordinates, and the zero set of this map is exactly the opposite Schubert variety corresponding the partition $(1,0, \ldots, 0)$. It therefore follows from basic properties of equivariant Chern classes \cite[Chapter 2, \S 3]{AF} that
    \[\sigma_1 = c_1^T\left(\left(\bigwedge\nolimits^n \Sc\right)^* \otimes \bigwedge\nolimits^n \Oc^n\right)=-c_1^T\left(\bigwedge\nolimits^n \Sc\right)+ c_1^T\left(\bigwedge\nolimits^n\Oc^n\right)=-e_T|_{G(n,m)}+t_1+\ldots+t_n\]
\end{proof}

The following Pieri rule for double Schur polynomials is standard and follows, for example, from the Littlewood-Richardson rule of \cite{MS}, but we include a proof here for completeness.

\begin{lemma}
\label{lem:pieri}
    For every partition $\lam$ with at most $n$ parts,
    \[(x_1+\ldots +x_n)\cdot s_\lam(\xs|\ts)=-\left(\sum_{i=1}^n t_{\lam_i+n-i+1}\right) s_\lam(\xs|\ts) + \sum_{\lam' \gtrdot \lam} s_{\lam'}(\xs|\ts),\]
    where the sum is over all $\lam'$ obtained by adding one box to the diagram of $\lam$.
\end{lemma}

\begin{proof}

Multiplying by $a_\rho$ on both sides, this is equivalent to showing that 
\[(x_1+\ldots +x_n)\cdot a_\nu(\xs|\ts)=-\left(\sum_{i=1}^n t_{\nu_i+1}\right) a_\nu(\xs|\ts) + \sum_{\nu' \gtrdot \nu} a_{\nu'}(\xs|\ts)\]

for all $\nu_1>\ldots>\nu_n\ge 0$, where the sum is over all $\nu'$ obtained from $\nu$ by increasing one of the $\nu_i$ by $1$ (if $\nu_i+1=\nu_{i-1}$, that term vanishes since $a_{\nu'}$ is then a determinant with two equal rows).

Using the identity $x_i\cdot (x_i|t)^k= (x_i|t)^{k+1}-t_{k+1}(x_i|t)^k$, we have 

\begin{align*}
    (x_1+\ldots +x_n)\cdot a_\nu(\xs|\ts) &= \sum_{j=1}^n \sum_{\sigma \in S_n} x_j \cdot \sgn(\sigma) \prod_{i=1}^n (x_{\sigma(i)}|\ts)^{\nu_i} \\
    &= \sum_{j=1}^n \sum_{\sigma \in S_n} x_{\sigma(j)}\sgn(\sigma) \prod_{i=1}^n (x_{\sigma(i)}|\ts)^{\nu_i} \\
    &= \sum_{j=1}^n \sum_{\sigma \in S_n} \sgn(\sigma) \left( -t_{\nu_j+1} \prod_{i=1}^n (x_{\sigma(i)}|\ts)^{\nu_i} + \prod_{i=1}^n (x_{\sigma(i)}|\ts)^{\nu_i+\delta_{ij}} \right)\\
    &= -\sum_{j=1}^n t_{\nu_j+1} a_\nu(\xs|\ts) + \sum_{\nu' \gtrdot \nu} a_{\nu'}(\xs|\ts),
\end{align*}
which is what we wanted to show.

\end{proof}

We can now prove

\begin{theorem}
    For every partition $\lam$ fitting in the $n$ by $m-n$ box, the map $\Phi$ sends $s_\lam(\xs|\ts)$ to $\sigma_\lam$.
\end{theorem}

\begin{proof}

For each $\lam$, let $\eps_\lam\in \{\pm 1\}$ be such that $\Phi(s_\lam)=\eps_\lam \sigma_\lam$. 

Since $\Phi$ is $\gf^\vee_{R_T}$-equivariant and preserves multiplicative identities, we have

\begin{equation} \label{eq:xk-act-on-1}
    (-e_T)^k \cdot 1_{\Lambda_{n,m}} \overset{\Phi}{\longmapsto} (-e_T)^k \cdot 1_{H_T^\bullet(G(n,m))}
\end{equation}
for all $k\ge 0$. 
The left hand side is $(x_1+\ldots +x_n)^k$ (using that $-e_T$ acts on $V$ via multiplication by $x$, hence on $\Lambda_{n,m}$ via multiplication by $x_1+\ldots+x_n$) which, when expanded in the double Schur basis, gives by Lemma \ref{lem:pieri}
\[ \sum_{|\lam|=k} f_\lam s_\lam(\xs| \ts) + \text{terms in the $s_\lam$'s for $|\lam|<k$},\]

where $f_\lam$ is the numbers of chains of partitions going from $\varnothing$ to $\lam$ where a single box is added at each step (equivalently, the number of standard Young tableaux of shape $\lam$). 

Now, $-e_T$ acts on $H_T^\bullet(G(n,m))$ via multiplication by $-e_T|_{G(n,m)}=\sigma_1-t_1-\ldots-t_n$ (by Lemma \ref{lem:res-eT}). Hence the right hand side of (\ref{eq:xk-act-on-1}) is

\[\left(\sigma_1-t_1-\ldots-t_n\right)^k= \sum_{|\lam|=k} f'_\lam \sigma_\lam + \text{terms in the $\sigma_\lam$'s for $|\lam|<k$},\]

where $f'_\lam$ is the coefficient of $\sigma_\lam$ in the expansion of $\sigma_1^k$, which lies in $\NN[t_1-t_2, \ldots, t_{n-1}-t_n]\subset R_T$ by Theorem \ref{thm:positivity}, but in fact in $\NN$ for degree reasons.

Now, we see from (\ref{eq:xk-act-on-1}) that $f_\lam'=\eps_\lam f_\lam$ for all $\lam$ of size $k$. Since $f_\lam, f'_\lam$ are both nonnegative and $f_\lam$ is in fact strictly positive (every partition has at least one standard Young tableau), the signs $\eps_\lam$ must be positive. Since $k$ was arbitrary, this completes the proof.

\end{proof}

\section{Acknowledgements}

The author would like to thank Joel Kamnitzer for his supervision during this research. The research was supported by an Undergraduate Student Research Award from the Natural Sciences and Engineering Research Council of Canada (NSERC) supplemented by the Fonds de recherche du Qu\'ebec - Nature and Technologies (FRQNT).

\bibliography{refs.bib}
\bibliographystyle{plain}

\end{document}